\documentclass[12pt]{article}

\usepackage{amsmath,amsfonts,amsthm}
\usepackage{amssymb}
\usepackage{euscript}
\usepackage{color}

\newtheorem{proposition}{Proposition}
\newtheorem{theorem}{Theorem}
\newtheorem{corollary}{Corollary}
\newtheorem{lemma}{Lemma}
\theoremstyle{definition}

\theoremstyle{definition}

\newtheorem*{Nakano}{Nakano Carrier Theorem}

\newcommand{\opoly}[2]{\EuScript{P\!}_o(^{#1}{#2})}

\DeclareMathOperator{\supp}{supp}

\newcommand{\R}{\ensuremath{\mathbb{R}}}

\newcommand{\N}{\ensuremath{\mathbb{N}}}

\newcommand{\MK}{\EuScript{M}(K)}

\title{A Nakano Carrier Theorem for Polynomials}

\author{Christopher Boyd, Raymond A. Ryan and Nina Snigireva}

\date{}

\makeatletter
        \def\@thefnmark{\null}
        \def\footnotetexta{\@footnotetext}

\begin{document}
\baselineskip=.65cm 
\maketitle

\begin{abstract}\footnotetexta{{\bf Keywords:} Banach lattice; Homogeneous polynomial; Orthogonally additive; Order continuity; Carrier.}
\footnotetexta{{\bf MSC(2020):} 46B42; 46G25.}
We use a localisation technique to study orthogonally additive polynomials
on Banach lattices. We derive alternative characterisations for orthogonal
additivity of polynomials and orthosymmetry of $m$-linear mappings. We
prove that
an orthogonally additive polynomial which is order continuous at one point is
order continuous at every point and we give an example to show that this result
does not extend to regular polynomials in general. Finally, we prove a Nakano
Carrier theorem for orthogonally additive polynomials, generalising a result
of Kusraev.
\end{abstract}
 
\section{Introduction}

The purpose of this paper is to prove a carrier theorem for homogeneous polynomials on Banach lattices that generalises the classical Nakano Carrier Theorem for linear functionals. 

Let us first recall the definitions of the null ideal and the carrier of a linear functional. Let $E$ be a Banach lattice and let $\varphi\in E'$.  
The null ideal (or absolute kernel) of $\varphi$ is
$$
N_\varphi = \bigl\{x\in E : |\varphi|(|x|)=0 \bigr\}\,.
$$
The carrier of $\varphi$ is the disjoint complement of its
null ideal:
$$
C_\varphi = N_\varphi^\perp = \bigl\{x\in E: x\perp y 
\text{ for every } y\in N_\varphi\bigr\}\,.
$$

Let us now state the Nakano Carrier Theorem, see for example \cite{ABur, Zaa}. We will give the statement of this theorem for Banach lattices, though it is valid in a more general setting. 
\begin{Nakano}
	Let $\varphi$ and $\psi$ be order bounded linear functionals on the Banach lattice $E$,
	one of which is order continuous. 
	 Then $\varphi \perp \psi$ if and only if $C_{\varphi} \perp C_{\psi}$.
\end{Nakano}

A natural question is how can we extend this to $m$-homogeneous polynomials on Banach lattices. Each $m$-homogeneous polynomial has an associated symmetric $m$-linear mapping on $E^m$. This complicates the definition of a carrier for an $m$-homogeneous polynomial as it is not clear if it should be a subset of $E$ or $E^m$. On the other hand, every $m$-homogeneous polynomial can be considered as a linear functional on a suitable tensor product. From this point of view, the carrier could be defined as a subset of a tensor product. However, the natural tensor product, which is the Fremlin tensor product, does not behave well with respect to order continuity. Namely, there is no perfect correspondence between order continuity of a multilinear map and the associated linear functional on the tensor product \cite{FremlinT}.

In this paper, we will see that if we restrict our attention to the class of orthogonally additive polynomials then the carrier is a natural, well behaved object. Let us note that the orthogonally additive polynomials correspond to orthosymmetric multilinear maps and that  Kusraev \cite{Kusr} has proved a carrier theorem for a special class of orthosymmetric multilinear maps. Here we take a different approach. In particular, our approach involves localisation to principal ideals. This allows us to consider orthogonally additive polynomials on $C(K)$ which in turn allows us to view these polynomials as measures.

As in the linear case, order continuity of polynomials will play an important role in the proof of the nonlinear version of Nakano's theorem. So we will discuss order continuous polynomials in detail. In particular, we will show that for orthogonally additive polynomials order continuity at the origin implies order continuity at every point. We will also, construct a counterexample to show that this is not the case for regular polynomials.  

For further information on polynomials we refer the reader to \cite{Dineen} and for the theory of Banach lattices to \cite{Aliprantis, Meyer-Nieberg}.

\section{Orthogonally additive polynomials and their properties}

Let $E$ be a Banach lattice. Recall that a regular $m$-linear
form $A\colon E^m \to \R$ is said to be
{orthosymmetric} if 
$A(x_1,\dots,x_m)=0$ whenever some pair $x_i$, $x_j$ 
of the arguments are disjoint, i.e. $|x_i|\wedge |x_j| = 0$. Note that it suffices to check this
condition for positive arguments, 
since $|x_i|\perp |x_j|$ if and only if
$x_i^\pm$ are  disjoint
from  $x_j^\pm$. Bu and Buskes \cite{BuBuskes12}
showed that every orthosymmetric $m$-linear form
is symmetric. Recall as well that a regular $m$-homogeneous
polynomial $P= \widehat{A}$ is said to be
orthogonally additive if $P(x+y)=P(x)+ P(y)$
whenever $x$ and $y$ are disjoint. Bu and Buskes \cite{BuBuskes12} showed that this is equivalent to the associated symmetric $m$-linear mapping $A$ being orthosymmetric.

Orthosymmetric $m$-linear mappings on $C(K)^m$ and orthogonally additive $m$-homogeneous polynomials on $C(K)$ have very useful integral representations which we will use throughout this paper. If we work in a principal ideal $E_a$ generated by a positive element $a$, then we can identify $E_a$ with the Banach lattice $C(K_a)$ for some compact Hausdorff space $K_a$ \cite{Kakut} and make use of these integral representations.

\begin{theorem}[\cite{Carando2006, PV}]\label{integral}
Let  $P$ be an orthogonally additive $m$-homogeneous polynomial on $C(K)$ and $A$ be the associated orthosymmetric $m$-linear mapping on $C(K)^m$. Then
$$
P(x)=\int_{K} x(t)^m\,d\mu(t)
$$
and therefore
$$
A(x_1,\dots,x_m)=\int_{K}x_1(t)\cdots x_m(t)\, d\mu(t),
$$
for some regular Borel measure $\mu$ on $K$.

\end{theorem}

Let us note that the above theorem was proved in \cite{PV} and also, with a shorter proof, in \cite{Carando2006}, for orthogonally additive $m$-homogeneous polynomials. The corresponding result for orthosymmetric $m$-linear mappings follows easily from the integral formulae for polynomials. For completeness, we will now give a short proof of this fact. Using the polarisation formula, we have
\begin{align*}
A(x_1,\dots,x_m)&=\frac{2^m}{m!}\sum_{\varepsilon_i=\pm 1}\varepsilon_1\cdots\varepsilon_m P(\varepsilon_1x_1+\cdots+\varepsilon_mx_m)\\
&=\frac{2^m}{m!}\sum_{\varepsilon_i=\pm 1}\varepsilon_1\cdots\varepsilon_m\int_{K}(\varepsilon_1x_1+\cdots+\varepsilon_mx_m)^m\,d\mu\\
&=\int_{K}\frac{2^m}{m!}\sum_{\varepsilon_i=\pm 1}\varepsilon_1\cdots\varepsilon_m(\varepsilon_1x_1+\cdots+\varepsilon_mx_m)^m\,d\mu\\
&=\int_{K}x_1\cdots x_m\,d\mu.
\end{align*}

The above theorem allowed us to show that there is a Banach lattice isometric isomorphism from the space $\MK$ of regular Borel measures on a compact Hausdorff space $K$, equipped with the variation norm, to the space $\opoly{m}{C(K)}$ of orthogonally additive polynomials on $C(K)$, equipped with the regular norm \cite{BRS}.

\begin{theorem}[\cite{BRS}]\label{p:isomorphism}
	Let $K$ be a compact, Hausdorff space and let $m$ be a positive integer.
	Let $I_m \colon \MK  \to \opoly{m}{C(K)}$
	be given by
	$$ 
	(I_m\mu)(x) = \int_K x^m \,d\mu\,.
	$$
	Then 
	$I_m$ is a
	Banach lattice isometric
	isomorphism from $\MK$ onto $\opoly{m}{C(K)}$.
\end{theorem}
Orthosymmetry can be characterised
using functional equations with no
restriction (such as disjointness) on
the arguments.  We begin by generalising the result for the bilinear case which was proved by Kusraev \cite{Kusr1}.
\begin{proposition}[\cite{Kusr1}]\label{Kus}
	A regular bilinear form $A$ on $E\times E$ is orthosymmetric if and only if
	\begin{equation}\label{bil}
	A(x,y) = A(x\vee y,x\wedge y)
	\end{equation}
	for all $x,y\in E$.
\end{proposition}

Now let $A$ be an orthosymmetric $3$-linear form.
Let $x,y,z\in E$.  By orthosymmetry with respect to the first two
variables, we have
$
A(x,y,z) = A(x\vee y,x\wedge y, z)$.
Now using orthosymmetry in the first and third 
variables,
\begin{align*}
A(x,y,z)&= A(x\vee y\vee z, x\wedge y, 
(x\vee y)\wedge z)\\
&= A(x\vee y\vee z, x\wedge y,
(x\wedge z)\vee(y\wedge z)\,,
\end{align*}
by the distributivity of the lattice
operations.
Finally, we use orthosymmetry with respect
to the second and third variables to get
$$
A(x,y,z)= 
A(x\vee y\vee z, (x\wedge y) \vee
(x\wedge z)\vee(y\wedge z),
(x\wedge y)\wedge (
(x\wedge z)\vee(y\wedge z)))\,.
$$
Now
\begin{align*}
(x\wedge y)\wedge (
(x\wedge z)\vee(y\wedge z))&=
((x\wedge y)\wedge (x\wedge z)) \vee
((x\wedge y)\wedge (y\wedge z))\\
&= x\wedge y\wedge z
\end{align*}
So we get the following identity that is 
satisfied by every orthosymmetric $3$-linear
form:
\begin{equation}\label{trilin}
A(x,y,z)= \\
A(x\vee y\vee z, (x\wedge y) \vee
(x\wedge z)\vee(y\wedge z),
x\wedge y \wedge z)
\end{equation}
for all $x,y,z\in E$. Conversely, if $A$
satisfies this, then clearly it is orthosymmetric.

One interesting consequence of this identity
is that if $A$ is an orthosymmetric $3$-linear
form, then in order to have $A(x,y,z)=0$,
it is not necessary for some pair of $x,y,z$
to be disjoint.  It suffices for them to be
``triply disjoint", in the sense that
$|x|\wedge |y| \wedge |z| = 0$.

We now consider the general case.  The pattern is clear, but we need some notation to avoid
very cumbersome formulae.  Let $ x=(x_1,\dots,
x_m)$ be an $m$-tuple in $E^m$.  
We define
\begin{align*}
J_1^m(x) &= \bigvee_{i=1}^{m}x_i\\
J_2^m(x)&= \bigvee_{\stackrel{1\le i_1,i_2\le m}{i_1<i_2}} \bigl(x_{i_1}\wedge x_{i_2}\bigr)\\
J_3^m(x)&= \bigvee_{\stackrel{1\le i_1,i_2,i_3\le m}{i_1<i_2<i_3}} \bigl(x_{i_1}\wedge x_{i_2}\wedge x_{i_3}\bigr)\\
\vdots &  \\
J_m^m(x) &= x_1\wedge \dots \wedge x_m
\end{align*}

Note that, if $x_1,\dots,x_m$
are real numbers, 
then $J_1^m(x), \dots, J_m^m(x)$ is the decreasing
rearrangement of $x_1,\dots,x_m$.
So we have
\begin{equation}\label{decrear}
\prod_{i=1}^m x_i = \prod_{i=1}^m J_i^m(x)\,.
\end{equation}
The next result shows that orthosymmetric
multilinear forms are characterised by a similar
identity. The $3$-linear case suggests that a natural way to prove this result would be to use induction. However, it turns out that an easier approach is to use a localisation technique.

\begin{proposition} Let $E$ be a Banach lattice.
	A regular $m$-linear form $A$ on 
	$E^m$ is orthosymmetric if and only if
	\begin{equation}
	\label{mlin}
	A(x_1,\dots,x_m) = 
	A(J_1^m(x),\dots, J_m^m(x))
	\end{equation}
	for every  $x= (x_1,\dots,x_m)
	\in E^m$. 
\end{proposition}

\begin{proof} Clearly if $A$ satisfies $A(x_1,\dots,x_m) = 
	A(J_1(x),\dots, J_m(x))$ then $A$ is orthosymmetric. To prove the converse we will use a localisation technique.  Let $a=|x_1|\vee\dots\vee|x_m|$. The required identity can now be easily established on $E_a$ since $E_a$ can be identified with $C(K_a)$ for a suitable compact set $K_a$.  Hence, using Theorem \ref{integral} and equation (\ref{decrear}), we have
	\begin{align*}
	A(x_1,\dots,x_m)& =\int_{K}x_1(t)\dots x_m(t)\, d\mu(t)&\\
	                &=\int_{K}J_1^m(x(t))\dots J_m^m(x(t))\, d\mu(t)&\\
	                &=\int_{K}J_1^m(x)(t)\dots J_m^m(x)(t)\, d\mu(t)&\\
	                &=A(J_1^m(x),\dots, J_m^m(x))&
	\end{align*}
\end{proof}

We now consider orthogonally additive polynomials. Schwanke \cite{Schwanke} showed that it suffices
to verify the orthogonal additivity property of a regular homogeneous polynomial for
positive arguments.

\begin{lemma}[\cite{Schwanke}]\label{pos}
	
	Let $P$ be an $m$-homogeneous polynomial on a  Riesz space $E$.  If $P$ is orthogonally
	additive on the positive cone
	of $E$, then $P$ is orthogonally additive.
\end{lemma}

Let us now show that in order for a homogeneous polynomial to be
orthogonally additive, it suffices for it to be orthogonally additive on the positive and negative parts of $x$ in $E$.

\begin{lemma}
	An  $m$-homogeneous polynomial $P$
	on a Riesz space $E$ is orthogonally
	additive if and only if
	\begin{equation}
	P(x) = P(x^+) + (-1)^m P(x^-)
	\end{equation}
	for every $x\in E$.
\end{lemma}
\begin{proof}
	The condition is clearly necessary.
	Conversely, suppose that $P$ satisfies
	this condition.  Let $x,y$ be disjoint
	positive elements of $E$.  Taking 
	$z_t=x-ty$ for $t>0$, we have $z_t^+=x$ and $z_t^-=ty$ and so 
	$P(z_t)=P(x-ty)= P(x)+(-1)^m t^mP(y)$. We also have that 
	$$
	P(z_t)=P(x-ty)=\sum_{k=0}^{m}\binom{m}{k}(-1)^kt^kA(x^{m-k}, y^k).
	$$
	These two identities for $P(z_t)$ imply that $A(x^{m-k}, y^k)=0$ for $0<k<m$. Therefore
	$$
	P(x+y)=\sum_{k=0}^{m}\binom{m}{k}A(x^{m-k}, y^k)=P(x)+P(y)
	$$ 
	and hence $P$ is orthogonally additive on $E_+$.
\end{proof}

We recall that a Riesz space is uniformly complete if every principal ideal is complete in the unit norm and that every Banach
lattice is uniformly complete. Buskes and Schwanke \cite{BuskSchw}, improving on a result
of Kusraeva, showed that a regular $m$-homogeneous
polynomial $P=\widehat{A}$ on a uniformly complete
Riesz space $E$ is orthogonally additive if and only 
if 
$$
P\bigl((x_1^m+\dots+x_k^m)^{1/m}\bigr)=
P(x_1)+\dots + P(x_k)
$$
for every $k\in \N$, $x_1,\dots,x_k\in E_+$,
and in addition
$$
P\bigl((x_1\dots x_m)^{1/m}\bigr)
= A(x_1,\dots,x_m)
$$
for all $x_1,\dots,x_m\in E_+$.
The expressions that appear on the left
hand side of these formulae are 
defined by the Krivine functional calculus. Recently, Schwanke \cite{Schwanke} showed that 
just one of these conditions suffices.
\begin{proposition}[\cite{Schwanke}]
	Let $P= \widehat{A}$ be a regular $m$-homogeneous polynomial
	on a uniformly complete Riesz space $E$.
	The following are equivalent:
	\begin{enumerate}
		\item[(a)]
		$P$ is orthogonally additive.
		
		\medskip
		\item[(b)]
		$
		P\bigl((x_1^m+\dots+x_k^m)^{1/m}\bigr)=
		P(x_1)+\dots + P(x_k)
		$\\
		for every $k\in \N$, $x_1,\dots,x_k\in E_+$.
		
		\medskip
		\item[(c)]
		$
		P\bigl((x^m+y^m)^{1/m}\bigr)
		= P(x)+ P(y)
		$\\
		for every $x,y\in E_+$.
		
		\medskip
		\item[(d)]
		$
		P\bigl((x_1\dots x_m)^{1/m}\bigr)
		= A(x_1,\dots,x_m)
		$\\
		for all $x_1,\dots,x_m\in E_+$.
		\end{enumerate}
	
\end{proposition}

A real valued function $f$ on a vector lattice $E$ is 
said to be a valuation if $f(x\vee y) + f(x\wedge y) = f(x)+f(y)$
for all $x,y\in E$.  Buskes and Roberts \cite{BusRobe} showed that
an $m$-homogeneous polynomial on a Riesz space is orthogonally
additive if and only if it is a valuation.
We give an alternative proof that demonstrates the power of the localisation technique.

\begin{proposition}[\cite{BusRobe}]
	A regular $m$-homogeneous polynomial
	$P$ on a Banach lattice $E$ is 
	orthogonally additive if and only if
	\begin{equation}\label{poly2}
	P(x\vee y) + P(x\wedge y)
	= P(x)+ P(y)
	\end{equation}
	for all $x,y\in E_+$.
\end{proposition}
\begin{proof}
	This condition is clearly sufficient.
	Conversely, suppose $P$ is orthogonally
	additive and let $x,y\in E_+$.
	Restricting to the principal ideal generated by $a=x\vee y$,
	we have an integral representation
	for $P$ on $E_a\cong C(K)$ given by a regular Borel
	measure $\mu$.  Thus
	\begin{align*}
	P(x\vee y)+ P(x\wedge y)
	&=\int_K (x\vee y)(t)^m+ (x\wedge y)(t)^m
	\,d\mu\\
	&=\int_K (x(t)\vee y(t))^m+ 
	(x(t)\wedge y(t))^m
	\,d\mu\\
	&= \int_K x(t)^m + y(t)^m\,d\mu
	=P(x)+ P(y)
	\end{align*}
\end{proof}

We note that this result has an obvious
generalisation to an identity
relating to any $k$-tuple
$(x_1,\dots,x_k)$ of positive elements, namely,
\begin{equation}
\sum_{i=1}^{k} P(J_ix) = \sum_{i=1}^{k}P(x_i)\,.
\end{equation}

Finally, let us note that the orthogonally additive $m$-homogeneous polynomials form a band in the Banach lattice of regular $m$-homogeneous polynomials \cite{BusRobe}. In particular, a regular $m$-homogeneous polynomial $P$ is orthogonally
additive if and only if $|P|$ is orthogonally additive.

\section{Localisation of multilinear mappings and polynomials to principal ideals}

We now develop further the technique of localisation used in the previous section. We will consider which properties of multilinear mappings and polynomials are preserved by localisation to principal ideals. In particular, we are interested whether order continuity of multilinear maps and polynomials is preserved by localisation. 

Let $E_1,\dots,E_m,F$ be Banach lattices. Recall that  $A\colon E_1\times E_2\times \dots \times E_m\to F$ is separately order continuous if it is order continuous in each variable when the other $m-1$ variables are kept fixed. Recall also that $A\colon E_1\times E_2\times \dots \times E_m\to F$ is jointly order continuous if  $$A(x^1_{\alpha_1}, \dots, x^m_{\alpha_m} )\stackrel{o}\to A(x_1, \dots, x_m) \quad \mbox{whenever $x^1_{\alpha_1}\stackrel{o} \to x_1, \dots,x^m_{\alpha_m}\stackrel{o} \to x_m $.}$$ Shotaev \cite{Shotaev} has shown that for bilinear maps, separate order continuity is equivalent to joint order continuity. We have extended this result to $m$-linear mappings in \cite{BRS2}. Let us note that a net $(x_\alpha)_\alpha$ in $E$  converges in order to $x$ if there exists a net $(y_\alpha)_\alpha$, decreasing to $0$,
and an index $\alpha_o$ so
that $|
x_\alpha-x|\le y_\alpha$ for all $\alpha\ge \alpha_o$ (see \cite{Abramovich2005}). Hence for all $\alpha\ge
\alpha_o$ we have that
\begin{align*}
|x_\alpha|&=|x_\alpha-x+x|\\
&\le |x_\alpha-x|+|x|\\
&\le y_\alpha+|x|\\
&\le y_{\alpha_o}+|x|.
\end{align*}
This tells us that if a net $(x_\alpha)_\alpha$ converges in order then we can find
$\alpha_o$ so that $(x_\alpha)_{\alpha\ge\alpha_o}$ is bounded in order and therefore a
norm bounded subset of $E$.

Let  $A\colon E^m \to F$ be an $m$-linear
mapping and let $a\in E_+$.  We will  denote by $A_a$
the 
restriction of $A$ to $(E_a)^m$. The next theorem shows which properties, including order continuity, of $A$ are characterised by restricting to principal ideals.
\begin{theorem}\label{locmulti}
	Let $E$, $F$ be Banach lattices, 
	with $F$ Dedekind complete and let $A,B\colon E^m \to F$ be $m$-linear
	mappings.  Then
	\begin{itemize}
		\item[(a)] $A$ is bounded if and only if
		$A_a$  is bounded on $E^m_a$ for every 
		$a\in E_+$.
		\item[(b)]
		$A$ is positive (respectively regular)
		if and only if $A_a$ is positive (respectively regular) on $E^m_a$ for every 
		$a\in E_+$.
		\item[(c)]
		Suppose that $A$ is regular. Then $A$ is jointly order continuous
		if and only if $A_a\colon (E_a)^m\to F$
		is jointly order continuous for every
		$a\in E_+$.
		
		\item[(d)] Suppose that $A$ is regular. Then 
		$$
		|A|_a = |A_a|, \quad
		(A\vee B)_a = A_a \vee B_a, \quad
		(A\wedge B)_a = A_a \wedge B_a
		$$
		for every $a\in E_+$.  In particular,
		$A\perp B$ if and only if
		$A_a \perp B_a$ for every $a\in E_+$.
	\end{itemize}
\end{theorem}

\begin{proof}
	(a) If $A$ is a bounded $m$-linear
	mapping, then clearly $A_a$ is a bounded 
	$m$-linear mapping on $E_a^m$
	for every $a\in E_+$.  
	
	To prove the converse, suppose that $A$ is unbounded.  Then there exist
	sequences $(x^1_n)$, $\dots$, $(x^m_n)$ in 
	the unit ball of $E$
	such that $\|A(x^1_n,\dots,x^m_n)\| \ge
	n^{m+1}$ for every $n$.  The sequences
	defined by $y^i_n = x^i_n/n$ converge
	to zero and so they have subsequences with the same indexing set that converge
	relatively uniformly to zero. In other words, the subsequences lie in some principal ideal $E_a$ and they converge to zero for the principal ideal norm.
	Therefore $A_a$ is unbounded for some 
	$a\in E_+$, a contradiction.
	
(b) Let us first note that an $m$-linear mapping $A$ is positive
if and only if $A_a$ is positive
on every principal ideal $E_a$.

Suppose that
every $A_a$ is regular. To show that $A$ is regular, by \cite{BuBuskes12}, we have to show
that the following supremum exists for each $m$-tuple
of positive elements of $E$:
$$
|A|(x_1,\dots,x_m) =
\sup\Bigl\{\sum |A(u^1_{i_1}, \dots , u^m_{i_m})|: u^i\in\pi(x_i), 1\le i\le m\Bigr\}.
$$
Note that if $x_1\vee\dots\vee x_m \le a$,
then this supremum defines $|A_a|$
at the $m$-tuple $(x_1,\dots,x_m)$
and its value is independent of $a$,
since replacing $a$ by a bigger element
does not add any additional partitions
of $x_1,\dots,x_m$. The converse is obvious.

(c) Observe
that if $(x_\alpha)$ is an order convergent
net in $E$, then there is an index $\alpha_0$
and a positive element $a$ in $E$ such
that the net $(x_\alpha:\alpha\ge \alpha_0)$
is order convergent in $E_a$. Suppose now that $A_a$ is jointly order continuous at $(x_1,\dots,x_m)$. Taking $\alpha_i\ge \alpha^i_o$ for $i=1,\dots,m$, we have that
$$
A(x^1_{\alpha_1}, \dots, x^m_{\alpha_m})=A_a(x^1_{\alpha_1}, \dots, x^m_{\alpha_m})\longrightarrow A_a(x_1, \dots, x_m)=A(x_1, \dots, x_m),
$$
showing that $A$ is jointly order continuous at $(x_1,\dots,x_m)$.
The converse direction is trivial.

(d) The argument in Part (b) also shows that 
$$
|A_a| = |A|_a 
$$
for every $a\in E_+$ with similar identities for $A^+$
and $A^-$. By the same reasoning, we have formulae
such as $(A\vee B)_a = A_a\vee B_a$.
And in particular, disjointness is also
determined on the principal ideals:
$A \perp B$ if and only if $A_a \perp
B_a$ for every $a\in E_+$. 
\end{proof}

To state the corresponding results for $m$-homogeneous polynomials, let us recall that a scalar valued $m$-homogeneous polynomial, $P$, is said to be order continuous at $x$ if $P(x_\alpha)$ converges to $P(x)$ whenever $(x_\alpha)_\alpha$ is a net in $E$
which converges in order to $x$. Moreover, it is easy to see that an $m$-homogeneous polynomial
is order continuous if and only if its
symmetric $m$-linear generator is jointly
order continuous.
	
\begin{proposition}\label{loc}
	Let $E$, $F$ be Banach lattices, 
	with $F$ Dedekind complete and let $P,Q\colon E \to F$ be $m$-homogeneous
	polynomials.  Then
	\begin{itemize}
		\item[(a)] $P$ is bounded if and only if
		$P_a$ (the restriction of $P$
		to $E_a$) is bounded on $E_a$ for every 
		$a\in E_+$.
		\item[(b)]
		$P$ is positive (respectively regular)
		if and only if $P_a$ is positive
		(respectively regular) on $E_a$ for every 
		$a\in E_+$.
		\item[(c)] Suppose that $P$ is regular. Then
		$P$ is order continuous if and only if
		$P_a$ is order continuous on $E_a$ for every 
		$a\in E_+$.
		
		\item[(d)] Suppose that $P$ and $Q$ are regular. Then 
		$$
		|P|_a = |P_a|, \quad
		(P\vee Q)_a = P_a \vee Q_a, \quad
		(P\wedge Q)_a = P_a \wedge Q_a
		$$
		for every $a\in E_+$.  In particular,
		$P\perp Q$ if and only if
		$P_a \perp Q_a$ for every $a\in E_+$.
	\end{itemize}
\end{proposition}

\section{Order continuous polynomials}

In this section we will focus on order continuous polynomials and discuss when order continuity at one point implies order continuity at every point. Let us begin by showing that it is possible for a polynomial on a Banach lattice $E$ to be
order continuous at $0$ yet not to be order continuous  at every point of $E$.
To see this suppose that $E$ is a Banach
lattice which admits an order continuous functional $\varphi$ and a norm
continuous linear functional $\psi$ which is not order continuous. For example, take $E=C[0,1]\bigoplus L_1[0,1]$, since there are no order continuous linear functionals on $C[0,1]$ and every bounded linear functional on $L_1[0,1]$ is order continuous.

Define the $m$-homogeneous polynomial $P$ by $P(x)=\varphi^{m-1}(x)\psi
(x)$. Let us first observe that $P$ is order continuous at $0$. To
see this let $(y_\alpha)_\alpha$ be a net in $E$ which converges to $0$ in
order. As we have already observed, we can assume without loss of generality that
$(y_\alpha)_\alpha$ is a (norm) bounded subset of $E$. Therefore, we can find $B>0$
so that $|\psi(y_\alpha)|\le B$ for all $\alpha$. Let $\varepsilon>0$.
As $(y_\alpha)_\alpha$ converges to $0$ in order
and $\varphi$ is order continuous we can find $\alpha_0$ so that
$|\varphi(y_\alpha)|^{m-1}<\frac{\varepsilon}{B}$ for all $\alpha>\alpha_o$. Then
for all $\alpha>\alpha_o$ we have
$$
|P(y_\alpha)|=|\varphi(y_\alpha)|^{m-1}|\psi(y_\alpha)|\le\frac{\varepsilon}{B}B
=\varepsilon
$$
which proves that $P$ is order continuous at $0$.

As $\psi$ is not order continuous, for any $x_o$ there exists
a net $(x_{\alpha})_\alpha $ which converges to $x_o$ in order but for
which $\psi(x_{\alpha})\not\to\psi(x_o)$. Let us suppose that $\varphi(x_o)\not
=0$.  As $\varphi$ is order continuous, we can find $A>0$ and
$\alpha_1$ so that $|\varphi(x_{\alpha})|^{m-1}>A$, all $\alpha>\alpha_1$.

We
consider two cases. In the first we
assume that $C=|\psi(x_o)|\not=0$. As $\psi(x_\alpha)\not\to\psi(x_o)$ we can
find $t>0$, a subnet $(x_{\alpha_\beta})_\beta$  of $(x_\alpha)_\alpha$ and
$\beta_2$ so that $|\psi(x_{\alpha_\beta})-\psi(x_o)|>t$ for all
$\beta>
\beta_2$. As $\varphi(x_{\alpha_\beta})$ converges to $\varphi(x_o)$ we can find
$\beta_3$ so that $|\varphi^{m-1}(x_{\alpha_\beta})-\varphi^{m-1}(x_o)|
<\frac{tA}{2C}$ for
all $\beta>\beta_3$. Let us choose $\beta_1$ so that $\alpha_{\beta_1}>\alpha_1$
and set $\beta_0=\max\{\beta_1,\beta_2,\beta_3\}$. Then
for $\beta>\beta_o$ we have
\begin{align*}
	|P(x_{\alpha_\beta})-P(x_o)| & =|\varphi^{m-1}(x_{\alpha_\beta})\psi(x_{\alpha_\beta})-
	\varphi^{m-1}(x_o)\psi(x_o)|\\
	&=|\varphi^{m-1}(x_{\alpha_\beta})
	\psi(x_{\alpha_\beta})-\varphi^{m-1}(x_{\alpha_\beta})\psi(x_o)-(\varphi^{m-1}
	(x_{o})\psi(x_o)-
	\varphi^{m-1}(x_{\alpha_\beta})\psi(x_o))|\\
	& \ge \left||\varphi^{m-1}(x_{\alpha_\beta})\psi(x_{\alpha_\beta})
	-\varphi^{m-1}(x_{\alpha_\beta}
	)\psi(x_o)|
	-|\varphi^{m-1}(x_o)\psi(x_o)-\varphi^{m-1}(x_{\alpha_\beta})\psi(x_o)|\right|\\
	& =\left||\varphi(x_{\alpha_\beta})|^{m-1}|\psi(x_{\alpha_\beta})
	-\psi(x_o)|-|\psi(x_o)||
	\varphi^{m-1}(
	x_{\alpha_\beta})-\varphi^{m-1}(x_o)|\right|\\
	& \ge tA- C\frac{tA}{2C}\\
	& =\frac{tA}{2}
\end{align*}                                                            
proving that $P$ is not order continuous at $x_o$.

Now let us suppose that $C=|\psi(x_o)|=0$. Then we can find $t>0$ and $\beta_2$
so that $|\psi(x_{\alpha_\beta})|>t$ for all $\beta>\beta_2$. Let $\beta_o=\max
\{\beta_1,\beta_2\}$. Then for $\beta>\beta_o$ we have
$$
|P(x_{\alpha_\beta})-P(x_o)|=|\varphi(x_{\alpha_\beta})|^{m-1}|\psi(x_{\alpha_\beta}
)|> tA
$$
and therefore $P$ is not order continuous at $x_o$.

While a regular homogeneous polynomial which is order
continuous at one point of a Banach lattice need not be order continuous
at every point we can still say something about the set of points
at which a regular polynomial is order continuous.

\begin{proposition}
	Let $E$ be a Banach lattice and $P$ a regular $m$-homogeneous
	polynomial on $E$. Let $C_{o}(P)$ denote the set of all points $x$ in $E$ for which $P$ is order continuous at $x$. Then $C_{o}(P)$ is a
	norm closed subset of $E$.
\end{proposition}

\begin{proof} Let $(u_n)_n$ be a sequence in $C_{o}(P)$ which converges
	to some point $u$ in $E$. Let $\varepsilon>0$. Let $(x_\alpha)_\alpha$ be a
	net in $E$ which is convergent to $0$ in order. Replacing $(x_\alpha)_\alpha$
	with a suitable tail, we may assume  without loss of generality,
	that $(x_\alpha)_\alpha$ is bounded. As $P$ is uniformly (norm)
	continuous on bounded sets there is $n_o$ in 
	$\mathbb{N}$ so that $|P(u)-P(u_n)|<\varepsilon/3$ and $|P(u+x_\alpha)-P(u_n+
	x_\alpha
	)|<\varepsilon/3$ for all $n\ge n_o$ and all $\alpha$. Since $u_{n_o}$ 
	belongs to $C_{o}(P)$, $P$ is order continuous at $u_{n_o}$.
	Therefore, we can find $\alpha_o$ so $|P(u_{n_o}+x_\alpha)-P(u_{n_o})|<
	\varepsilon/3$ whenever $\alpha>\alpha_o$. Then for $\alpha>\alpha_o$ we have
	\begin{align*}
		|P(u+x_\alpha)-P(u)|&\le |P(u+x_\alpha)-P(u_{n_o}+x_{\alpha})|+|P(u_{n_o}
		+x_{\alpha})-P(u_{n_o})|+|P(u_{n_o})-P(u)|\\
		&<\frac{\varepsilon}{3}+\frac{\varepsilon}{3}+
		\frac{\varepsilon}{3}\\
		&=\varepsilon.
	\end{align*}
	Hence we have that $P$ is order continuous at $u$.
\end{proof}

This allows us to show that order continuity of a homogeneous
polynomial at any point implies order continutiy at $0$.

\begin{corollary}\label{regocts0octsevery}
	Let $E$ be a Banach lattice and $P$ is a regular $m$-homogeneous polynomial
	on $E$. If $P$ is order continuous at some point of $E$ then
	$P$ is order continuous at $0$.
\end{corollary}

\begin{proof} Let us suppose that $P$ is order continuous at
	$x$. By homogeneity it follows that $P$ is order continuous
	at $\lambda x$ for every
	$\lambda$ in $\mathbb{R}$. Letting $\lambda$ tend to $0$ the above Proposition
	gives us that $P$ is order continuous at $0$.
\end{proof}

Let us now look at order continuity of orthogonally additive polynomials. We will show that in contrast to regular polynomials in general, order continuity of an orthogonally additive homogeneous polynomial
at any point implies order continuity at every point. We begin by looking at order continuity on $C(K)$-spaces. In this setting, we will make use of the following lemma.

\begin{lemma}\label{ocxm} If $(x_\alpha)_\alpha$ is a net in $C(K)$ which converges in order to $x$ then the net $(x^m_\alpha)_\alpha$ converges in order to $x^m$.
\end{lemma}

\begin{proof}
Suppose that $(x_\alpha)_\alpha$ is a net in $C(K)$ which converges in
order to $x$. Replacing $(x_\alpha)_\alpha$ with $(x_\alpha)_{\alpha\ge\alpha_o}$
we may assume without loss of generality that $(x_\alpha)_\alpha$
is order bounded. Let $B=\sup_\alpha\|x_\alpha\|$. 

First suppose that $x=0$. Then we have
$$
|x_\alpha^m|\leq B^{m-1}|x_\alpha|
$$
giving us that $(x_\alpha^m)_\alpha$ converges to $0$ in order. 

Now suppose that $x\neq 0$. In this case we have
\begin{align*}
|x_\alpha^m-x^m|&=\left|(x_\alpha-x)\left(\sum_{j=0}^{m-1}x_\alpha^{m-1-j}x^j
\right)\right|\\
&\le|x_\alpha-x|\sum_{j=0}^{m-1}|x_\alpha|^{m-1-j}|x|^j\\
&\le  |x_\alpha-x|\sum_{j=0}^{m-1}\|x_\alpha\|^{m-1-j}|\|x\|^j
1_K  \\
&\le mB^{m-1}|x_\alpha-x|,
\end{align*}
from this we conclude that $(x_\alpha^m)_\alpha$ converges to $x^m$ in order.
\end{proof}

In order to state our next result, we need to recall the definition of a normal measure. A regular Borel measure on a compact Hausdorff
space $K$ is said to be normal
if $\mu(B)=0$ for every closed, nowhere dense subset
$B$ of $K$. It is easy to see that $\mu$ is normal
if and only if the positive measure
$|\mu|$ is normal. Our next theorem is in the spirit of the following result of Dixmier \cite{Dix} which relates order continuous linear functionals and normal measures.  

\begin{theorem}[\cite{Dix}]
	Let $K$ be a compact Hausdorff space.  
	A bounded linear functional on $C(K)$ is 
	order continuous if and only if the (unique)
	regular Borel measure that represents it
	is a normal measure.
\end{theorem}

We now have the following theorem for polynomials.

\begin{theorem}\label{normal} Let $P$ be an orthogonally additive
	$m$-homogeneous
	polynomial
	on $C(K)$. Then $P$ is order continuous at $0$ if and only if its representing measure,
	$\mu$, is normal. 
\end{theorem}

\begin{proof} We can assume without loss of generality
	that $P$, and hence $\mu$, is positive.
	Let $B$ be a closed nowhere dense subset of $K$.
	Using Urysohn's Lemma we find a bounded decreasing net $(x_\alpha)_\alpha$ of 
	nonnegative continuous functions on $K$,
	indexed by the collection of open subsets
	of $K$ containing $B$, that converges pointwise to $\chi_B$.
	
	Then the infimum of the set $\{x_\alpha\}$
	is the zero function.
	To see this, suppose there exists a nonnegative
	function $y$ such that $0<y\le x_\alpha$ for
	every $\alpha$.  Applying this at each point, 
	we see that $0<y\le \chi_B$.
	But now the open set $\{t\in K: y(t)<1\}$
	is contained in $B$, a contradiction, 
	since $B$ is nowhere dense.
	
	Therefore, the zero function is the greatest
	lower bound of $\{x_\alpha\}$ as claimed and so
	$x_\alpha\searrow 0$. It follows that $x_\alpha^{1/m}\searrow 0$ and since
	$P$ is order continuous at $0$ we have
	$$
	\mu(B)= \int_K \chi_B\,d\mu \le
	\int_K x_\alpha\,d\mu =\int_K (x^{1/m}_\alpha)^m\,d\mu= P(x_\alpha^{1/m}) \to 0\,,
	$$
	and so $\mu(B)=0$. Therefore $\mu$ is a normal 
	measure.

To prove the converse, suppose that $(x_\alpha)_\alpha$ is a net in $C(K)$ which converges in
order to $0$. By Lemma \ref{ocxm} and since $\mu$ is a normal measure,  we have that
$$
P(x_\alpha)=\int_K x_\alpha^m\, d\mu\longrightarrow 0. 
$$
Hence $P$ is order continuous at $0$.

\end{proof}

The next proposition is a consequences of the above theorem.

\begin{proposition}\label{octeverypoint} Let $P$ be an orthogonally additive
	$m$-homogeneous polynomial on $C(K)$. Then $P$ is order continuous at $0$ if and only if $P$ is order continuous at
	every point of $C(K)$.
\end{proposition} 

\begin{proof}
Suppose that $P$ is order continuous at $0$ then Theorem \ref{normal} gives us that its representing measure, $\mu$, is normal. Let $(x_\alpha)_\alpha$ be a net in $C(K)$ which converges in order to $x$. Using Lemma \ref{ocxm} and the fact that $\mu$ is a normal measure, we have that
$$
P(x_\alpha)=\int_K x_\alpha^m\, d\mu\longrightarrow \int_K x^m\,d\mu=P(x),
$$
proving that $P$ is order continuous at every point of $C(K)$.

The converse is trivial. 
\end{proof}

Let us extend this result to arbitrary Banach lattices.

\begin{proposition}
	Let $E$ be a Banach lattice and $P$ be an orthogonally additive
	$m$-homogeneous polynomial on
	$E$ which is order continuous at some point of $E$. Then $P$ is order continuous at every
	point of $E$.
\end{proposition}

\begin{proof}  By Corollary \ref{regocts0octsevery}, $P$ is order continuous at $0$. Let $a\in E_+$.  
	We first observe that $P_a$ is orthogonally additive
	on $E_a$.
	Let $(y_\alpha)$ be net in $E_a$ which converges to
	$0$ in order in $E_a$. Since $E_a$ is an ideal in $E$, it follows that $(y_\alpha)_\alpha$ converges to $0$ in order in
	$E$. So we have that
	$P_a(y_\alpha)=P(y_\alpha)$ converges to $0$, proving that $P_a$ is order
	continuous at $0$. From Proposition \ref{octeverypoint}, it follows that $P_a$ is order
	continuous at each point of $E_a$. Finally, by Proposition \ref{loc} (c), we have that $P$ is order continuous at each point of $E$.
	
\end{proof}

\section{A carrier theorem for polynomials} 

Let $E$, $F$ be Banach lattices, with $F$ Dedekind complete. Following Kusraev \cite{Kusr}, we define the null ideal $N(A)$ of a regular symmetric $m$-linear mapping $A\colon E^m \to F$ by 

$$
N(A)=\left\{ x\in E \, :\, |A|(|x|,\dots,|x|)=0\right\}.
$$

Now if $P$ is the regular $m$-linear polynomial associated with $A$  then we define the null ideal $N(P)$ of $P$ by
$$
N(P)=\left\{ x\in E \, :\,|P|(|x|)=0\right\}=\left\{ x\in E \, :\, |A|(|x|,\dots,|x|)=0\right\}=N(A).
$$
Then the carrier $C(P)$ of $P$ and the carrier $C(A)$ of $A$ are defined by
$$
C(P)=N(P)^\perp
=N(A)^\perp=C(A).
$$

Let us recall that  if $P$ is a regular $m$-homogeneous
polynomial on $C(K)$, then it corresponds to 
a regular Borel measure $\mu$ on $K^m$. More precisely, Fremlin \cite{FremlinA} proved the following result, 
which readily extends to regular $m$-homogeneous
polynomials on $C(K)$: If $P$ is a regular $2$-homogeneous polynomial on $C(K)$ then there exists a regular Borel measure $\mu$ on $K^2$ such that 
$$
P(x)=\int_{K^2}x(s)x(t)\,d\mu(s,t)
$$
for all $x\in C(K)$.

For the linear functional on $C(K^m)$ defined by $\mu$, the the null ideal of $\mu$ is given by

$$
N(\mu) =
\bigl\{x\in C(K^m): \int_{K^m}|x|\,d|\mu| =0\bigr\} \,.
$$

We note that since $N(P)$ is a subset of  
$C(K)$ and $N(\mu)$ is a subset of 
$C(K^m)$, we cannot relate the null ideal of $P$ directly to the null ideal of the corresponding measure $\mu$.

However, if $P$ is orthogonally additive then the null ideal of $P$ coincides with null ideal of the corresponding measure $\mu$ as given by Theorem \ref{integral}. We note that in this case the null ideal of $P$ is given by
$$
N(P)=\left\{x\in C(K):\int |x|^m\,d|\mu|=0\right\}
$$
and the null ideal of $\mu$ is given by
$$
N(\mu) =\bigl\{x\in C(K): \int_{K}|x|\,d|\mu| =0\bigr\} \,.
$$
So clearly $N(P)=N(\mu)$.

Next let us recall that the support of a regular Borel measure $\mu$ on a compact set $K$, denoted by $\supp \mu$, is the smallest closed set $S$ such that $|\mu| (K\setminus S)=0$.

\begin{proposition}
Let $P\colon C(K) \to \R$ be an orthogonally additive $m$-homogeneous polynomial and $\mu$ the corresponding  regular Borel measure on $K$. Then we have
$$
N(P)=\left\{x\in C(K): x=0\quad |\mu|{\rm -a.e}.\right\}.
$$ Therefore,
$$
C(P)=\bigl\{x\in C(K):\supp{x}\subset \supp{\mu}\bigr\}.
$$
\end{proposition}

\begin{proof}
	It suffices to note that
	\begin{align*}
		N(P)&=\left\{x\in C(K):|P|(|x|)=0\right\}\\
		&=\left\{x\in C(K):\int |x|^m(t)\,d|\mu|(t)=0\right\}\\
		&=\left\{x\in C(K): x=0\quad |\mu|{\rm-a.e.}\right\}.
	\end{align*}
	We now have that
	$$
	C(P)=N(P)^\perp=\bigl\{x\in C(K):\supp{x}\subset \supp{\mu}\bigr\}. 
	$$
\end{proof}
	As we have noted earlier, $N(P)=N(\mu)$. We now observe that
	$$
	C(P)=N(P)^\perp=N(\mu)^\perp=C(\mu).
	$$

We can now state a Nakano carrier theorem for orthogonally additive polynomials on $C(K)$.

\begin{theorem}\label{nakanock}	Let $P$ and $Q$ be orthogonally additive $m$-homogeneous polynomials on $C(K)$. If $P$ is order continuous then $P\perp Q$ if and only if $C(P)\perp C(Q)$.
\end{theorem}

\begin{proof}
	Let $\mu_P$ and $\mu_Q$ be the representing measures for $P$ and $Q$ respectively as given by Theorem \ref{integral}. Recall that Theorem \ref{normal} and Proposition \ref{octeverypoint} give us that if $P$ is order continuous if and only if $\mu_p$ is normal.  
	Now by Theorem \ref{p:isomorphism}, we have $P\perp Q$ if and only if $\mu_P\perp \mu_Q$. Nakano's theorem for linear functionals gives us that 
	$\mu_P\perp \mu_Q$ if and only if $C(\mu_P) \perp C(\mu_Q) $. By definition $C(\mu_P)=C(P)$ and $C(\mu_Q)=C(Q)$ and the result follows.
\end{proof}

To extend the above result to arbitrary Banach lattices, we need  the following lemma.

\begin{lemma}\label{loccarr}
	Let $E$, $F$ be Banach lattices, 
	with $F$ Dedekind complete and let $P, Q\colon E \to F$ be regular $m$-homogeneous polynomials. Then 
	$$
	N(P_a) = N(P) \cap E_a
	\quad\text{and}\quad 
	C(P_a) = C(P)\cap E_a\,,
	$$
	for every $a\in E_+$. In particular, $C(P)\perp C(Q)$ if and only if $C(P_a)\perp C(Q_a)$ for every $a\in E_+$.
\end{lemma}
\begin{proof}
	Suppose that $x\in N(P_a)$. Then $x\in E_a$. Now if we consider $x$ as an element of $E$ and use Proposition \ref{loc} (d), we have
	$$
	|P|(|x|)=|P|_a(|x|)=|P_a|(|x|)=0
	$$
	Hence $x\in N(P)$ and so $N(P_a) \subset N(P)\cap E_a$.

	Next let $x\in N(P)\cap E_a$. Then
	$$
	|P_a|(|x|)=|P|_a(|x|)=|P|(|x|)=0
	$$ 
	Hence $x\in N(P_a)$ and so $ N(P)\cap E_a \subset N(P_a)$.

	For the second identity, suppose
	that $x\in C(P)\cap E_a$.  Let $y\in N(P_a)$. Then $y\in N(P)$ and so $x\perp y$. That is, $x\in N(P_a)^\perp=C(P_a)$.
	
	To establish the reverse inclusion, suppose that $x\in C(P_a)$
	but $x\not\in C(P)$.  Then there exists $y\in N(P)$ 
	such that $|x|\wedge |y| >0$.  Since $|x|\wedge |y|\leq |x|$, we have $|x|\wedge |y| \in E_a$. Hence
	$$
	|P_a|(|x|\wedge |y|)=|P|_a(|x|\wedge |y|)=|P|(|x|\wedge |y|)\leq |P|(|y|)=0
	$$
	and so $|x|\wedge |y|\in N(P_a)$. We conclude that $x\perp (|x|\wedge |y|)$. However, $|x|\wedge  (|x|\wedge |y|)=|x|\wedge |y|\neq 0$, a contradiction.
	
	Next we prove that $C(P)\perp C(Q)$ if and only if $C(P_a)\perp C(Q_a)$ for every $a\in E_+$. So suppose that $C(P)\perp C(Q)$. Since by the above $C(P_a)\subseteq C(P)$ and $C(Q_a)\subseteq C(Q)$, we see that $C(P_a) \perp C(Q_a)$.
	
	Now assume that $C(P_a) \perp C(Q_a)$ for every $a$ in $E_+$. Then  $C(P)\cap E_a\perp  C(Q)\cap E_a$ for every $a\in E_+$. Let $x\in C(P)$ and $y\in C(Q)$. Now take any $a>|x|\wedge |y|$.  Then $x, y \in E_a$.  So $x\in C(P)\cap E_a=C(P_a)$ and $y\in C(Q)\cap E_a=C(Q_a)$. Hence, by our assumption $x\perp y$ and so $C(P)\perp C(Q)$.

\end{proof}
 
Combining Lemma \ref{loccarr} and Theorem \ref{nakanock}, we have the following Nakano carrier theorem for orthogonally additive polynomials.
	
\begin{theorem}	
	Let $P$ and $Q$ be orthogonally additive $m$-homogeneous polynomials on a Banach lattice $E$. If $P$ is order continuous then $P\perp Q$ if and only if $C(P)\perp C(Q)$.
	
\end{theorem}
	
\begin{proof} 
By Proposition \ref{loc} (d), we have $P\perp Q$ if and only if $P_a\perp Q_a$ for every $a\in E_+$. By part (c) of Proposition \ref{loc} we also have that $P$	is order continuous if and only if $P_a$ is order continuous for every  $a\in E_+$. Now Theorem \ref{nakanock} gives us that $P_a\perp Q_a$ if and only if $C(P_a) \perp C(Q_a) $ for every  $a\in E_+$. Finally Lemma \ref{loccarr} tells us that $C(P_a) \perp C(Q_a) $ for every  $a\in E_+$ if and only if $C(P)\perp C(Q)$. 
\end{proof}

\bibliographystyle{plain}
\bibliography{Carrier}

\noindent Christopher Boyd, School of Mathematics \& Statistics, University
College Dublin, \hfil\break
Belfield, Dublin 4, Ireland.\\
e-mail: christopher.boyd@ucd.ie

\medskip

\noindent Raymond A. Ryan, School of Mathematics, Statistics and Applied 
Mathematics, National University of Ireland Galway, Ireland.\\
e-mail: ray.ryan@nuigalway.ie

\medskip

\noindent Nina Snigireva, School of Mathematics \& Statistics, University
College Dublin, \hfil\break
Belfield, Dublin 4, Ireland.\\
e-mail: nina@maths.ucd.ie

\end{document}